\documentclass[10pt]{amsart}

\usepackage{amsaddr}
\usepackage{amsmath,mathtools,amsthm,amsfonts,amssymb,color,verbatim,graphicx}

\usepackage[letterpaper]{geometry}
\usepackage{stmaryrd}

\newtheorem{theorem}[equation]{Theorem}

\newtheorem{proposition}[equation]{Proposition}

\theoremstyle{definition}

\numberwithin{equation}{section}

\begin{document}
\setlength{\jot}{0pt} 
\title{Bounds on the number of conjugacy classes of the symmetric and alternating groups}

\today
\author{Bret Benesh}
\address{
Department of Mathematics,
College of Saint Benedict and Saint John's University,
37 College Avenue South,
Saint Joseph, MN 56374-5011, USA,
bbenesh@csbsju.edu
}
\author{Cong Tuan Son Van}
\address{
Department of Mathematics,
Kansas State University,
138 Cardwell Hall,
1228 N. 17th Street,
Manhattan, KS 66506-2602, USA,
congvan@math.ksu.edu
}

\begin{abstract}
Let $G$ be a finite group with Sylow subgroups $P_1,\ldots,P_n$, and let $k(G)$ denote the number of conjugacy classes of $G$.  Pyber asked if $k(G) \leq \prod_{i=1}^n k(P_i)$ for all finite groups $G$.  With the help of GAP, we prove that Pyber's inequality holds for all symmetric and alternating groups.
\end{abstract}

\maketitle

\begin{section}{Introduction}
L. Pyber submitted Problem 14.71 to the Kourovka Notebook~\cite{Kourovka}, which reads:  ``Let $k(H)$ denote the number of conjugacy classes of a group $H$, and $G$ be a group with Sylow $p$-groups $P_1,\ldots,P_n$.  Prove or disprove:  $k(G) \leq k(P_1)\cdot\ldots\cdot k(P_n)$."    We will let $k_p(G)$ denote $k(P)$ for a Sylow $p$-subgroup $P$ of $G$, which means that Pyber's inequality can be restated as $k(G) \leq \prod_{p\bigm||G|}k_p(G)$.  We will verify that this inequality holds if $G$ is a symmetric group $S_n$ or alternating group $A_n$.  Our strategy is to use estimates from analytic number theory to show that that Pyber's inequality holds for $n \geq 60,000$ and  a GAP~\cite{GAP} calculation to show the result holds for $n < 60,000$.

\end{section}

\begin{section}{Main Results}
Let $p(n)$ denote the number of partitions of an integer $n$. It is well-known that $k(S_n)=p(n)$ and $k(A_n) \leq 2k(S_n) = 2p(n)$ (see~\cite[Section~11.1]{Scott}, for instance). Additionally, every Sylow $p$-subgroup $P$ of $S_n$ is a Sylow $p$-subgroup of $A_n$ if $p$ is odd, and every Sylow $2$-subgroup of $S_n$ has order twice that of every Sylow $2$-subgroup of $A_n$.   

Now let $[n]_p$ denote the largest power of $p$ that divides $n$ for a positive integer $n$ and prime $p$.   The first proposition is a bound due to Hall, and the second proposition is the result of a simple GAP script.

\begin{proposition}\cite[Chapter~V.15.2]{Huppert}\label{prop:pbound}
Let $P$ be a finite group of order $p^{2m+e}$ for some prime $p$, nonnegative integer $m$, and $e \in \{0,1\}$.  Then $k(P) \geq p^e+(p^2-1)m$ 
\end{proposition}

\begin{proposition}\label{prop:SmallnMain}
If $n < 60,000$, then $2p(n) \leq \prod_{p \bigm| \frac{n!}{2}} \left(p^{e_p}+(p^2-1)m_p\right)$,  where $[\frac{n!}{2}]_p=p^{2m_p+e_p}$  for some $m_p$ and $e_p$ with $e_p \in \{0,1\}$ for all primes $p$ dividing $\frac{n!}{2}$. 
\end{proposition}

\begin{proposition}\label{prop:SmallnResult}
If $n < 60,000$, then $k(S_n) \leq \prod_{p \bigm| n!}k_p(S_n)$ and $k(A_n) \leq \prod_{p \bigm| \frac{n!}{2}}k_p(A_n)$.
\end{proposition}
\begin{proof}
The result is easy to check if $n\in\{1,2,3\}$, so assume that $4 \leq n < 60,000$.  The right-side of the inequality from Proposition~\ref{prop:SmallnMain} exactly bounds $\prod_{p \bigm| |A_n|}k_p(A_n)$ by Proposition~\ref{prop:pbound}, and so we have $k(A_n) \leq 2p(n) \leq \prod_{p \bigm| |A_n|}k_p(A_n)$ for all $n < 60,000$.

Because the bounds in Proposition~\ref{prop:pbound} are only a function of the order of the Sylow subgroup, the lower bound from Proposition~\ref{prop:pbound} for $k_2(A_n)$ is also a lower bound for $k_2(S_n)$.  Then we use the notation and result from Proposition~\ref{prop:SmallnMain} and the fact that Sylow $p$-subgroups of $A_n$ have the same order as Sylow $p$-subgroups of $S_n$ for odd $p$ to get 

\begin{align*}
k(S_n)  &< 2p(n) \\
&\leq \prod_{p \bigm| (n!/2)}(p^{e_p}+(p^2-1)m_p)\\
&\leq (2^{e_2}+(2^2-1)m_2)\prod_{\substack{p \bigm| (n!/2) \\  p\text{ odd prime}}}(p^{e_p}+(p^2-1)m_p)\\
&\leq k_2(S_n)\prod_{\substack{p \bigm| n! \\ p\text{ odd prime}}}(p^{e_p}+(p^2-1)m_p)\\
&\leq k_2(S_n)\prod_{\substack{p \bigm| n! \\ p\text{ odd prime}}}k_p(S_n)\\
&\leq  \prod_{p \bigm| n!}k_p(S_n).
\end{align*}

\end{proof}

It remains to show that the result holds for $n \geq 60,000$.


\begin{proposition}\label{prop:BignResult}
If $n \geq 60,000$, then $k(S_n) \leq \prod_{p\bigm|n!}k_p(S_n)$ and $k(A_n) \leq \prod_{p\bigm|\frac{n!}{2}}k_p(A_n)$.
\end{proposition}
\begin{proof}
The center of a nontrivial $p$-group is nontrivial, so we have $k_p(G) \geq p$ for all primes $p$ dividing $|G|$.  Thus, we have \[\prod_{p\bigm|n!}k_p(S_n) \geq \prod_{p\bigm|n!}p \geq 2^{\pi(n)},\] where $\pi(n)$ is the number of primes that are at most $n$.  Similarly, $\prod_{p\bigm|\frac{n!}{2}}k_p(A_n) \geq 2^{\pi(n)}$.  Further, $\pi(n) \geq \frac{n}{6\log n}$ by~\cite[Theorem~4.6]{Apostol}, so $\prod_{p\bigm|n!}k_p(S_n)$ and $\prod_{p\bigm|\frac{n!}{2}}k_p(A_n)$ are both bounded below by $2^{\frac{n}{6\log n}}$. 

As previously stated, both $k(S_n)$ and $k(A_n)$ are bounded above by $2p(n)$. By~\cite[Theorem~14.5]{Apostol}, $p(n) < e^{\pi\sqrt{\frac{2n}{3}}}$, so we have both $k(S_n)$ and $k(A_n)$ are at most $2e^{\pi\sqrt{\frac{2n}{3}}}$.  It remains to show that $2e^{\pi\sqrt{\frac{2n}{3}}}$ is at most $2^{\frac{n}{6\log n}}$.  We let \[f(n)=\left(2^{\frac{n}{6\log n}}\right)\left(e^{\pi\sqrt{\frac{2n}{3}}}\right)^{-1}.\]  It is easy to check that $f(60,000) \approx 5.45 \geq 2$ and $f$ is increasing for $n \geq 60,000$, so $f(n) \geq 2$ and \[2e^{\pi\sqrt{\frac{2n}{3}}} \leq 2^{\frac{n}{6\log n}}\] for $n \geq 60,000$.  We conclude that if $G$ is $S_n$ or $A_n$, then \[k(G) \leq 2p(n) \leq 2e^{\pi\sqrt{\frac{2n}{3}}} \leq 2^{\frac{n}{6\log n}} \leq 2^{\pi(n)} \leq \prod_{p\bigm||G|}k_p(G).\]

 
\end{proof}

Propositions~\ref{prop:SmallnResult} and~\ref{prop:BignResult} imply our final theorem.

\begin{theorem}
If $n$ is any positive integer, then $k(S_n) \leq \prod_{p\bigm|n!}k_p(S_n)$ and $k(A_n) \leq \prod_{p\bigm|\frac{n!}{2}}k_p(A_n)$.
\end{theorem}

\end{section}
\newpage
\bibliographystyle{amsplain}
\bibliography{BibCong} 

\providecommand{\bysame}{\leavevmode\hbox to3em{\hrulefill}\thinspace}
\providecommand{\MR}{\relax\ifhmode\unskip\space\fi MR }
\providecommand{\MRhref}[2]{%
  \href{http://www.ams.org/mathscinet-getitem?mr=#1}{#2}
}
\providecommand{\href}[2]{#2}
\begin{thebibliography}{1}

\bibitem{Apostol}
T.~M. Apostol, \emph{Introduction to analytic number theory}, Springer-Verlag,
  New York-Heidelberg, 1976, Undergraduate Texts in Mathematics.

\bibitem{GAP}
The GAP~Group, \emph{{GAP -- Groups, Algorithms, and Programming, Version
  4.6.4}}, 2013.

\bibitem{Huppert}
B.~Huppert, \emph{Endliche {G}ruppen. {I}}, Die Grundlehren der Mathematischen
  Wissenschaften, Band 134, Springer-Verlag, Berlin-New York, 1967.

\bibitem{Kourovka}
V.~D. Mazurov and E.~I. Khukhro (eds.), \emph{The {K}ourovka notebook},
  seventeenth ed., Russian Academy of Sciences Siberian Division, Institute of
  Mathematics, Novosibirsk, 2010, Unsolved problems in group theory, Including
  archive of solved problems.

\bibitem{Scott}
W.~R. Scott, \emph{Group theory}, second ed., Dover Publications, Inc., New
  York, 1987.

\end{thebibliography}

\end{document}